\documentclass[11pt]{article}
\usepackage[usenames]{color}
\usepackage{graphicx}
\usepackage{amscd}
\usepackage[colorlinks=true,
linkcolor=webgreen, filecolor=webbrown, citecolor=webgreen]{hyperref}
\definecolor{webgreen}{rgb}{0,.5,0}
\definecolor{webbrown}{rgb}{.6,0,0}
\usepackage{color}
\usepackage{fullpage}
\usepackage{float}
\usepackage{amsmath}
\usepackage{amssymb}
\usepackage{amsthm}
\usepackage{amsfonts}
\usepackage{latexsym}
\usepackage{epsf}

\begin{document}

\begin{center} \vskip 1cm {\LARGE\bf  The open  polynomials of the finite topologies   }  \vskip 1cm \large
Moussa Benoumhani\\
Al-Imam University\\
Faculty of Sciences, Department of Mathematics\\
P.O.Box 90950 Riyadh 11623 Saudi Arabia\\
\href{mailto:mbenoumhani@kku.edu.sa}{\tt benoumhani@yahoo.com  }
\end{center}

\begin{abstract}
Let $\tau $ be a topology on the finite set $X_n$. We consider the
open polynomial associated with the topology $\tau$. Its
coefficients are the
cardinality of open sets of size  $j=0,\;\ldots ,\;n.$\\
 J. Brown \cite{JB} asked when this polynomial has only real zeros. We prove
that this polynomial has real zeros, only in the trivial case where
$\tau$ is the discrete topology. Then, we weaken Brown's question:
for which topology this polynomial is log--concave, or at least
unimodal? A  partial answer is given. Precisely, we prove that if
the topology has a large number of open sets, then its open
polynomial is unimodal.
 \end{abstract}
\numberwithin{equation}{section}

\setlength{\textwidth}{7.2in} \setlength{\oddsidemargin}{.1in} \setlength{\evensidemargin}{.1in}
\setlength{\topmargin}{-.5in} \setlength{\textheight}{8.9in}

\newtheorem{theorem}{Theorem}[section]
\newtheorem{acknowledgement}[theorem]{Acknowledgement}
\newtheorem{corollary}[theorem]{Corollary}
\newtheorem{definition}[theorem]{Definition}
\newtheorem{lemma}[theorem]{Lemma}
\newtheorem{proposition}[theorem]{Proposition}
\newtheorem{remark}[theorem]{Remark}
\newtheorem{conclusion}[theorem]{Conclusion}
\newtheorem{example}[theorem]{Example}

\section{Introduction}
Let $X= X_n$ be an $n$-- element set. A topology $\tau \subseteq
\mathcal{P}(X)$ on a set $X$ is just a family of subsets
 (called open sets) containing $X, \; \phi$, and is closed under union and finite
 intersection. A base or a basis for $\tau$ is subset of $\tau$,
 such that each open set is a union of members of the basis. We need
 the following definitions in finite topologies.
  \begin{definition}
  A partition $\displaystyle \pi= \left ( A_1, A_2, \; ...,\; A_{l} \right)$ of the set $X$ is a collection of subsets of $X$ such that all the $A_i$ are nonempty,
  $\displaystyle  A_i \bigcap A_j =\phi,\; { \rm for} \; i\neq j,$  and  $X=\displaystyle \bigcup_{i=1}^{l} A_i. $
  \end{definition}
   \begin{definition}
  A partition $\displaystyle \pi= \left ( A_1, A_2, \; ...,\;A_{l}\right )$ of the finite set $X_n$ is called of type
  $ \left( \alpha_1,\; \alpha_2,\;\ldots, \; \alpha_l \right ) $   if it contains $ \alpha_i$ subsets of cardinality $i$, $($note that $ \sum_{i}i\alpha_i=n$ $)$.
  \end{definition}
  A partition topology is given in the following definition
   \begin{definition}
  A partition topology $\tau$  on $X$, is a topology that is induced by a
  partition of $X$.
  \end{definition}
  For example the topology  $\tau= \{ \phi,\;\{ a,\;b\},\; \{c,\;d\},\;\{ e\},\; \{a,\;b\;,c,\;d\},\;\{ a,\;b\;,e\},\; \{c,\;d\;,e\},\; X \}$ on  $X=\{ a,\; b\;,c,\;d,\;e
  \}$ is induced by the partition $\displaystyle \pi= \left( \{a,\;b \},\;\{c, \;d\},\;\{e\}\right
  )$. \\
  In general the union of two topologies on a set is not a topology. But, there is a kind of union to remedy to this lack: the disjoint union topology
\begin{definition}
 Let $X$,  $Y$ be two topological spaces, such that $X\bigcap Y=\phi$. The disjoint union topology  on $Z=X\bigcup Y$, is
 defined by
 $$\tau= \left\{ U \; such \; that \; U \in \tau_X \;{\rm  or} \; U \in \tau_Y \; {\rm or}\;U=U_X\bigcup U_Y \right \}$$
  \end{definition}
When the space $X$ is finite, and is a disjoint union of two spaces
$Y_1$ and $Y_2$, then its topology  $\tau$ satisfies
 $|\tau |=|\tau_{Y_1}|\cdot |\tau_{Y_2}|$. This result will be helpful for us in the
next sections. \\
 Let $\tau$ be a topology on $X_n$. Let us designate by $u_j$ the number of open sets
in $\tau$ having cardinality $j$, and consider the polynomial
 $P(x)=\displaystyle \sum_{j=0}^{n}u_{j}x^j$, called the {\it open set polynomial} of
 $\tau$. \\
 J. Brown \cite{JB}  raised the following question: \\
 {\it When are all of the roots of an open set polynomial real}? \\
 In the next section, we answer this question by showing that the only case where this happens is for the discrete
 topology. Before doing this, we recall some facts about polynomials and log--concave sequences. \\
 A real positive sequence $(a_j)_{j=0}^{n}$ is said to be  { \it unimodal }, if there exist integers $k_0, \; k_1,\; k_0\leq k_1$,( integers $k_0\leq j\leq k_1$  are the modes of the sequence) such that
 $$ a_0 \leq a_1\leq \ldots \leq a_{k_0} =a_{k_0+1}=\ldots=a_{k_1}\geq a_{k_1+1} \geq \ldots \geq a_n.$$ It is { \it log-concave} if
 $\displaystyle  a_j^2\geq a_{j-1}a_{j+1},\; {\rm for} \; 1\leq j \leq n-1.$
 A real sequence $(a_i)$ is said to be  with no internal zeros (NIZ), if $ i<j, a_i\neq 0, \; a_j\neq 0 \; {\rm then}\; a_l \neq 0 $ for every $l, i\leq l \leq j$. A (NIZ) log-concave sequence is obviously unimodal, but the converse is not true. The sequence 1, 1, 4, 5, 4, 2,1 is unimodal but not log-concave. Note the importance of (NIZ): the sequence 0, 1, 0, 0, 2, 1 is  log-concave but not unimodal.
A real polynomial is unimodal (log-concave, symmetric, respectively) provided that the sequence of its coefficients is unimodal (log-concave, symmetric, respectively).\\
If inequalities in the log-concavity definition are strict, then the
sequence is called strictly log-concave (SLC for short), and in this
case, it has at most two consecutive modes. The following result may
be helpful in proving unimodality:
\begin{theorem}
 If the polynomial $ \displaystyle \sum_{j=0}^{n}a_jx^j $ associated with the sequence  $(a_j)_{j=0}^{n}$ has only real zeros then
    \begin{equation} a_j^2 \geq \frac{j+1}{j}\frac{n-j+1}{n-j} a_{j-1}a_{j+1},\; {\rm for} \; 1\leq j \leq n-1 \end{equation}
    \end{theorem}
When dealing with sequences, we manipulate their sum, product,...
etc. So, we may ask about the preservation of these properties.
Recall that the convolution product of the sequences
 $ \displaystyle (a_j)_{j=0}^n $ and $ \displaystyle (b_j)_{j=0}^n$ is just the sequence $ \displaystyle (c_j) $, such that $
\displaystyle c_k=\sum_{i+j=k}a_jb_i $. It is known that the
convolution of two log-concave sequence is
 a log-concave one, but the convolution of two unimodal sequence is not
 necessarily  unimodal. However,  we need the following result
\begin{theorem}
The convolution two log-concave sequences is
 a log-concave sequence.\\
The convolution of a unimodal and
 a log-concave sequence is a unimodal one.
    \end{theorem}
\noindent  Since we are concerned by real polynomials with positive coefficients, recall the simple facts:\\
    1- If there is $a_j=0$ for a ceratin $j, 1  \leq j \leq n-1$ then the polynomial can not have all its zeros real. \\
    2- If the inequalities (1,1) fail for  a ceratin $j, 1 \leq j \leq n-1$, then the roots of the polynomial can not be all real.\\
  \section{ The roots of the open polynomial are  not real}
  Our first result is the
    \begin{theorem}
  The open set polynomial $ \displaystyle \sum_{j=0}^{n}u_jx^j $ has only real zeros if and only if $\tau$ is the discrete  topology on  $ X_n$.    \end{theorem}
  \begin{proof}
  If the topology is discrete on $X_n$, then  $ \displaystyle \sum_{j=0}^{n}u_jx^j=(1+x)^n. $
  Now, suppose that there exist a topology $\tau$, such that its open  polynomial has only real zeros. Note the following facts: \\
  1)  $ u_0=u_n=1$.\\
  2)  $\displaystyle  0 \leq u_j \leq \binom{n}{j}, \; 1 \leq j \leq   n-1.$\\
   3) All the $u_i$ must be  $ \neq 0$, otherwise the polynomial could not have all its zeros real.\\
   Now,  since the polynomial has only real zeros, its coefficients satisfy Newton inequalities $(1.1)$, which  may be written,
   $$ \frac{u_{j}}{u_{j-1}} \geq \frac{j+1}{j} \frac{n-j+1}{n-j} \frac{u_{j+1}} {u_{j}},\;  \; 1\leq j \leq n-1.$$
    By induction, we deduce  $$ \frac{u_{1}}{u_{0}} \geq n^{2} \frac{u_{n}}{u_{n-1}}.$$
    Remembering that  $u_0=u_n=1$, we obtain
    $$ u_{1}u_{n-1} \geq n^2.$$ But, we also have   $$ u_{1}u_{n-1} \leq n^2. $$
    The $u_k$ are integers, it follows that
    $ u_1=u_{n-1}=n$. This means that the topology is the discrete one.\end{proof}
    \begin{remark}
    In the proof of the last theorem, if we suppose just that the polynomial satisfies inequalities $(1.1)$, we reach the same conclusion,
    that is; the topology is the discrete one. This means that log--concave
    open polynomials are rare, or equivalently, almost the only
    log-concave open set polynomial is that one associated with the
    discrete topology.
    \end{remark} The following result illustrates the previous remark.
    \begin{theorem}
  If  the open set  polynomial $ \displaystyle \sum_{j=0}^{n}u_jx^j $ satisfies $ u_j^2 \geq du_{j+1}u_{j-1}$ for
  $d>1$, and $\tau$ is not the discrete topology
  on $X_n$, then $\displaystyle d< n^{\frac{2}{n-1}}$.
    \end{theorem}
  \begin{proof}
The inequalities may be written,
   $$ \frac{u_{j}}{u_{j-1}} \geq d \frac{u_{j+1}} {u_{j}},\; {\rm for} \; 1\leq j \leq n-1, $$
    then   $$ \frac{u_{1}}{u_{0}} \geq d^{n-1} \frac{u_{n}}{u_{n-1}}.$$
    Because  $u_0=u_n=1$, and $ \displaystyle n \geq a_1,\; a_{n-1}$, it yields
    $\displaystyle  d^{n-1} \leq n^2,$ or equivalently $$ d \leq  n^{\frac{2}{n-1}}. $$
This means that if $n$ is large enough, the only log--concave open
set polynomial is that of the discrete topology. \end{proof}
\noindent The natural question after this last result is a weak version of Brown's question: \\
  \begin{center} {\it Which topologies have their open set polynomials log-concave or at least
   unimodal ?} \end{center}
 In order to investigate this question, we have to know, in a first time, the topologies for which all the coefficients of the open polynomial
  are non vanishing. Unfortunately, this is not always an obvious fact. In some cases, it is possible to know explicitly the polynomial, in others, the topology
  itself let some information about the coefficients of the open polynomial. Recall this definition from \cite{Ko}
 \begin{definition}
 Let $\tau$ be a topology on the finite set $X$. The nonempty open set $A$, is
  called a minimal open set if
  either  $A \bigcap U= \phi$ or $ A \bigcap U=A $ for every $U \in \tau$.
  \end{definition}
  The following lemma is obvious but important
  \begin{lemma}
  If $\tau$ is a topology on $X_n$, then so is $ \displaystyle
  \tau^{c}= \left \{ U^c  \;  / \; U \in \tau \right
  \}$, called the cotopology of $\tau$.
  \end{lemma}
  \noindent It may happen that $ \tau =\tau^c$. For example, let
  $\displaystyle X_4=\left \{ a, b, c, d \right \}$ and   $
  \displaystyle \tau= \left \{ \phi, \{a, b \}, \{c,d \}, X_4 \right
  \}$, then  $ \tau=\tau^c$.
  We may ask for which topology this is true. The next lemma
  characterizes topologies with this property
\begin{lemma}
  Let  $\tau$ be a topology on $X_n$, then
  $$  \tau =\tau^{c}\Longleftrightarrow \tau \quad  is \;
  induced\; by\; a\; partition $$
  \end{lemma}
\begin{proof} If $\tau $ is
  induced by a partition, then it is obvious that $\tau^c=\tau$. Now suppose that $\tau^c=\tau$. Consider the minimal open sets $A_i,\; 1 \leq i \leq l$
   of $\tau$. We have for $i\neq j$, $A_i\bigcap A_j= \phi$ and $ \displaystyle X_n=\bigcup _{i=1}^{l}A_i$. If  $ \displaystyle X\neq \bigcup _{i=1}^{l}A_i$, then
   $ \displaystyle \phi \neq X_n-\bigcup _{i=1}^{l}A_i=C=A_{l+1} \in \tau$ is another
   minimal open set. Contradiction. So,
    $ \displaystyle \left ( A_1, \; A_2, \;\ldots, \; A_l \right ) $ is a  partition of $X_n$ \end{proof}
  In the case where $\tau$ is induced by a partition, its open polynomial is given explicitly
  \begin{theorem}
  Let $\tau$ be a topology induced by the partition $\pi$ of type $ \left( \alpha_1,\; \alpha_2,\;\ldots, \; \alpha_l \right ) $, then
  $$ P_{\tau}(x)=\prod_{i=1}^{l}\left (x^{i}+1\right )^{\alpha_i}$$
   \end{theorem}
   \begin{proof}
   The base of the topology contains $\alpha_i$ subset of cardinality $i$. So, it follows that the number of open sets of cardinality $j$ is indeed
    the possible combinations of all the different unions of the members of the base,
     this is also the coefficient of the polynomial in the right side of the previous formula. \end{proof}
   For example, the open polynomial associated with  the topology
induced  by the partition $\displaystyle \pi= \left( \{a,\;b
\},\;\{c, \;d\},\;\{e\}\right ) $ is $\displaystyle
P(x)=(x+1)(x^2+1)^2=1+x+2x^2+2x^3+x^4+x^5,$ while the polynomial
associated with the partition $\displaystyle \pi= \left( \{a,\;b
\},\;\{c, \;d,\; e\} \right ) $ is $\displaystyle
P(x)=(x^2+1)(x^3+1)=1+x^2+x^3+x^5.$ In the last example, there are
some missing coefficients. Let the  open polynomial $ \displaystyle
\sum_{j=0}^{n}u_jx^j $, be  associated with a topology induced by a
partition, when have we $u_j\neq 0$ for all $ 0 \leq j \leq n$? It
is obvious that $\alpha_1$ must be $ \geq 1$. So, we will suppose
this condition.
\begin{theorem}
  Let $\tau$ be a topology induced by the partition $\pi$ of type $ \left( \alpha_1 ( \geq 1),\; \alpha_2,\;\ldots, \; \alpha_l \right ) $, then
 the coefficients of  $$ \displaystyle P_{\tau}(x)=\prod_{i=1}^{l}\left (x^{i}+1\right
 )^{\alpha_i}=\sum_{j=0}^{n} u_jx^{j},$$  are given by
 $$ u_m=\sum_{j_1+2j_2+ \ldots +lj_l=m}\binom{\alpha_1}{j_1}\cdots \binom{\alpha_l}{j_l}.$$
 The coefficient $u_m$ is $\neq 0$ if and only if, for every $1 \leq j \leq n-1$,  the equation  $ x_1+2x_2+\ldots
 +lx_l=m$ has a  solution $ \left( j_1,\; j_2,\;\ldots, \; j_l \right ) $ such
 that at least one term  $\binom{\alpha_1}{j_1}\cdots
 \binom{\alpha_j}{j_l}\neq 0.$
   \end{theorem}
   \begin{proof}
   The coefficient $u_j$ is obtained by an identification between
   the left and right expression.
   \end{proof}
   Our last result in this section, gives the maximum
   cardinality of a topology lacking an open set of a prescribed
   cardinality.
\begin{theorem}
  Let $\tau$ be a topology on $X_n$. If $\tau$ has no open set with cardinality $j, \; 1\leq j\leq
  n-1$, the cardinality of this topology can not exceed
   $ \displaystyle 2^{j-1}+2^{n-j-1}$
   \end{theorem}
\begin{proof}
Let $\tau$ be a topology on $X_n$,  having no open set of
cardinality $j, \; 1 \leq j \leq n-1$. The topology $\tau$  will
have at most $j-1$ singletons. So, $2^{j-1}$ open sets of
cardinality $\leq j-1$. Now take all sets of cardinality $\geq j+1$,
containing the set of cardinality $j-1$. Their number is $
2^{n-j-1}$. The total number of open sets of this topology is as
claimed, $ \displaystyle 2^{j-1}+2^{n-j-1}$.
\end{proof}
As a consequence we have the
\begin{corollary} If the topology $ \tau$ exists and is such that $|\tau| \geq
2^{n-2}+2$, then all the coefficients of its open polynomial are non
vanishing.
\end{corollary}
Unfortunately, even if the topology has a large number of open sets,
this does not guarantee the unimodality!!! For example, consider the
topology $\tau$ on $X_n$
$$\tau=  \mathcal{P}(X_{n-2}) \bigcup  \{ \{x_1, x_2, \;\ldots, \; x_{n-2},\; x_{n-1}\},\; \{x_1, x_2, \;\ldots, \; x_{n-2}, \; x_{n}\}, \; X_n \}$$
We have $ \displaystyle |\tau|=2^{n-2}+3$, and its open polynomial
is $$ P(x)=(x+1)^{n-2}+2x^{n-1}+x^n$$
is not unimodal. \\
 In conclusion of this section, we list some
cases where the coefficients of the polynomial $ \displaystyle
\sum_{j=0}^{n}u_jx^j$ are non vanishing . \\
1) If the topology $ \tau$ is $T_0$. \\
2) If the topology is induced by a partition, and  the number of
open sets, which are singletons is large enough
(for example $ \displaystyle \geq \frac{n}{2}$). \\
3) If $ \tau$ is a topology such that $ |\tau| \geq 2^{n-2}+2$.\\
4) It is easy to see that if the topology $\tau$ is $T_0$ and $
|\tau|=n+1,$ or $ n+2$, then the sequence is unimodal. Since in this
case, we have either a
chain or a topology with the diagram as above.\\
 Finally, note that
if   the $ |\tau| \leq n$, then the polynomial can not be unimodal,
there is at least $0<j_0<n$, such that $u_{j_0}=0$.
\begin{center}
$$ \unitlength=1truemm
\begin{picture}(40,40)(25,-13)
 \put(39.25,0){$\bullet$}
 \put(40,1){\line(-1,1){10}}
\put(29.25,10){$\bullet$} 
 \put(30,11){\line(1,1){10}}
 \put(39.25,20){$\bullet$}
 \put(39,23){$X_n$}
 \put(40,21){\line(1,-1){10}}
 \put(49.25,10){$\bullet$}
\put(40,1){\line(1,1){10}}
 \put(39.8,-15){\vdots}
  \put(39.8,-20){\vdots}
 \put(39,-39){$\emptyset$}
\put(39.25,-10){$\bullet$}
 \put(39.2,-28){$\bullet$}
 \put(39.2,-21){$\bullet$}
\put(40,-35){\line(0,1){15}}
 \put(40,-9){\line(0,1){10}}
\put(39.25,-36){$\bullet$}
 \put(40,-35){\line(0,1){15}}
\par\vspace{3cm}
 \put(30,-43){Figure 1}
\end{picture}
$$
\end{center}
\par\vspace{2.5cm}

\section{A Class of unimodal open polynomials}

Intuitively, the open set polynomial will be unimodal, if the
topology has a large number of open sets (so, a large number of open
sets which are singletons, this means that the topology looks like
the discrete one). In this section, we will determine explicitly the
open set polynomials of a class of topologies having a large number
of open sets ($\tau$ such that $|\tau| \geq 6\cdot 2^{n-4} $ and
another class of topologies $\tau$ such that $|\tau| \geq 5\cdot
2^{n-4} $). We will also examine the unimodality of
these polynomials.\\
{\bf Convention} \\
The symbol $X_k$ will designate any finite set of cardinality $k$.
If we consider a subset of $X_n$ of cardinality $k \leq n-1$, the
elements of $X_n \setminus X_k$ will be designated by $a, b, c,...$,
or $ y_1,\; y_2, \; ...$ if the set is large,
 those of $X_k$ are denoted by $x_1,\; x_2,\; x_3, \;...$ \\
 Consider the set $ \tau(n,k)$ of the
topologies having $k$ open sets. In order to compute its
cardinality, Kolli \cite{Ko} divided this set into two
  disjoint sets: \begin{center} $\displaystyle  \tau_1(n,k)=\{ \tau\in \tau(n,k): \; \bigcap_{U  \in \tau  }U \ne \phi, \; U \;{\rm nonempty \; set} \} \quad {\rm
  and} \quad \tau_2(n,k)=\tau(n,k)-\tau_1(n,k).$ \end{center}
   For $ \tau \in \tau_1(n,k)$ and $ k \geq
  5.2^{n-4}$, all  open set polynomials are given in the following theorem.
\begin{theorem}
 Let  $\displaystyle \tau  \in \tau_1(n,k) $ with
 $ k \geq 5\cdot2^{n-4}$ open sets. Then the open polynomial of $\tau$ is
 one of the three
\begin{eqnarray*}
       P_1(x) & =& x(x+1)^{n-1}+1  \\
       P_2(x) &= & x^2(x+1)^{n-2}+x(x+1)^{n-3}+1 \\
       P_3(x) & = & x^3(x+1)^{n-3}+2x^2(1+x)^{n-4}+x(1+x)^{n-4}+1. \\
     \end{eqnarray*}
   \end{theorem}
\begin{proof}
If $\displaystyle \tau  \in \tau_1(n,k) $  and $ k \geq
5\cdot2^{n-4}$, then $\tau$ is necessarily one of
the three kind of the following topologies (see \cite{Ko}): \\
1) The cotopology of $\tau= \mathcal{P}(X_{n-1}) \bigcup  \{X_n\} $, whose open polynomial is  $ (x+1)^{n-1}+x^n$. \\
 2) The cotopology of  $\tau= \mathcal{P}(X_{n-2})\bigcup \{ \{ a, x\},\; X_n\} $, its open polynomial
 is $$ x^2(x+1)^{n-3}+(x+1)^{n-2}+x^n. $$
 3) The cotopology of  $\tau= \mathcal{P}(X_{n-3})\bigcup\{  \{ a, x\}, \{ b,
 x\}, \;  X_n\}.$  Its open polynomial is
 $$x^3(x+1)^{n-3}+2x^2(1+x)^{n-4}+x^3(1+x)^{n-4}+x^n.$$
 Note that if the open polynomial of the topology $\tau$ is $ \displaystyle \sum_{j=0}^{n}u_jx^j $, then
 $ \displaystyle \sum_{j=0}^{n}u_{n-j}x^j $ (its reciprocal), is the open polynomial of
 $\tau^{c}$. Furthermore, the wanted
 polynomials are just the reciprocal of these last ones:
$$P_1(x)=x(x+1)^{n-1}+1, \quad P_2(x)= x^2(x+1)^{n-2}+x(x+1)^{n-3}+1,$$
and
$$P_3(x)=x^3(x+1)^{n-3}+2x^2(1+x)^{n-4}+x(1+x)^{n-4}+1.$$
 \end{proof}
 \begin{corollary}
The polynomials $P_1, \; P_2 \;P_3$ are unimodal for every
  $n \geq 3$, but not log-concave.
\end{corollary}

The remaining of this section is devoted to the determination of the
polynomials of the topologies in the set $\displaystyle  \tau_2(n,k)
$ and $ k \geq 6\cdot2^{n-4}$. Kolli \cite{Ko} computed these
topologies according to the number and cardinal  of their minimal
open sets. We use these details to see the geometry of the topology,
then to determine $u_j$, then its  open polynomial. We start by the
topologies having $(n-1)$ minimal open sets.
\begin{theorem}
If $\tau$ is a topology having  $(n-1)$  minimal open sets, at least one is not a singleton, then: \\
Either  $\tau$ is induced by a partition, thus its open set
polynomial is
 $$ P_1(x)  = (1+x^2)(x+1)^{n-2}.$$
\noindent Or  $\tau=\mathcal{P}(X_{n-3})\bigsqcup \{ \phi, \;
\{a,b\},\; \{a, b, c\} \}$   (the disjoint union topology) and its
open polynomial is given by
$$ P_2(x)=(1+x^2+x^3)(1+x)^{n-3}. $$
The polynomials$ P_1, \; P_2$  are log-concave.\\
 If the minimal open sets
of $\tau $ are all singletons, then all the open polynomials
associated with these topologies are given by
      $$ P_l(x)  = (x+1)^{n-1}+x^{l+1}(1+x)^{n-l-1}, \; l=1,\;\ldots, n-1. $$
 These polynomials are unimodal for all $1 \leq l \leq n-1$.
   \end{theorem}
   \begin{proof}
This topology see \cite {Ko} is such that $|\tau|=2^{n-1}\quad
or\quad  6\cdot2^{n-4}$. In the first case, it is induced by a
partition,
 whose blocks are all singletons, but one have cardinality two. So, its open polynomial is
\begin{eqnarray*}
       P_1(x) & =& (1+x^2)(1+x)^{n-2}  \\
       P_1(x) &= & (1+x+x^2+x^3)(1+x)^{n-3}
     \end{eqnarray*}
 In the second case, $\tau$
 is the disjoint union  of the discrete topology on $(n-3)$ elements and the chain topology $\{\phi,\; \{a,b\},\;
 \{a,b,c\} \}.$  The open polynomial is then
\begin{eqnarray*}
       P_2(x) & =&  (1+x^2+x^3)(1+x)^{n-3} \\
       P_2(x) &= & (1+x^2+x^3)(1+x)^3(1+x)^{n-6} \\
       P_2(x) &= &  \left ( 1+3x+4x^2+5x^3+6x^4+4x^5+x^6 \right )(1+x)^{n-6}
     \end{eqnarray*}
 Now, since $P_1,\; P_2$ are product  of two log-concave
 polynomials, it follows that they are log-concave too.
  In the case where the topology has $(n-1)$ singletons $\{ x_j \}, \; 1 \leq j \leq n-1$, as minimal open sets,  only one element of $X_n$ is not open, $a=x_n$, say.
 Any topology having $(n-1)$ singletons as minimal open sets, is obtained by adjoining a subset of the form
   $\{ x_1, x_2, \;...,\; x_{l},\; a \}, \; 1 \leq l \leq n-1$.
Its cardinality is $2^{n-1}+2^{n-l-1}, \; 1 \leq l \leq n-1.$
   The number of open sets of cardinality $i$ is $u_{i}=\binom{n-1}{i}, 0 \leq i \leq
   l$   and   $u_{l+i}=\binom{n-1}{l+i}+ \binom{n-l-1}{i-1}, 1 \leq i \leq n-l-1$
and then the open set polynomials are as stated. Also, These
polynomials are unimodal for all $1 \leq l \leq n-1$, because
 $$ P_l(x)  = (x+1)^{n-1}+x^{l+1}(1+x)^{n-l-1}= (x+1)^{n-l-1} \left ( x^{l+1}+(1+x)^{l} \right ). $$
 These polynomials are unimodal,  since they are the convolution product of a log-concave and a unimodal
 sequence.
\end{proof}
Now, we consider the open set  polynomials of topologies, in the
same range as above ($ |\tau| \geq 6\cdot2^{n-4}$) with $(n-2)$
minimal open sets, all of them are singletons. The idea is the same,
since all the minimal open sets are singletons, we form these
topologies by adjoining  suitable subsets, to the discrete topology
on $X_{i}$, to obtain the desired topologies in the range $ \geq
6\cdot 2^{n-4}$. So many cases are to be considered.
\begin{theorem}
 If $\tau$ is a topology obtained from  $(n-2)$ singletons as minimal open sets, and the adjoining of
 $\{a,\;x\} \; {\rm and} \; A_j=\{a,b,x, x_1, x_2,\;\ldots, \; x_j \}, \; 0 \leq j \leq n-3$,
then $|\tau|=6\cdot 2^{n-4}+2^{n-3-j}, \; 1 \leq j \leq n-3$.
 and the open polynomials associated with
 these topologies are given by
 $$ P_{j}(x)  =(x+1)^{n-2}+x^2(x+1)^{n-3}+ x^{j+3}(1+x)^{n-3-j}, \quad 0 \leq j
\leq n-3  $$
   \end{theorem}
 \begin{proof} This is case (a) in page 5 of \cite{ Ko}. The the adjoining  of
 $ \{a,\; x\}$  contributes $2^{n-3}$ open sets, while
 the adjoining of $A_j=\{a,b,x, x_1, x_2,\;\ldots, \; x_j \}, \; 0 \leq j
\leq n-3$, add $2^{n-j-3}$ open sets to the topology
$\mathcal{P}(X_{n-2})$, so the open set polynomials of these
topologies are
 \begin{eqnarray*}
P_{j}(x) & =&(x+1)^{n-2}+x^2(x+1)^{n-3}+ x^{j+3}(1+x)^{n-3-j} \\
          &=& (1+x)^{n-3-j} \left ( x^{j+3}+x^2(x+1)^{j}+(x+1)^{j+1}
          \right )
\end{eqnarray*}
\end{proof}
\begin{theorem}
 If $\tau$ is a topology obtained from  $(n-2)$ singletons as minimal open sets, with  the adjoining of $\{a,\;x_1,\;x_2\}$ and $\{a,\;b,\; x_1,\;x_2\}$
  then  the open set polynomial associated with
 this  topology is
 $$ P(x)  =(x+1)^{n-2}+x^3(x+1)^{n-4}+ x^4(1+x)^{n-4}=(1+2x+x^2+x^3+x^4)(1+x)^{n-4}.$$
   \end{theorem}
 \begin{proof} This is case (b) in page 5 of \cite{ Ko}. The adjoining
 of the open sets $\{a,\;x_1,\;x_2\}$ and $\{a,\;b,\; x_1,\;x_2\}$ contributes, each one,
 with $2^{n-4}$ open sets. The wanted polynomial is
 polynomials $$  P(x)  =(x+1)^{n-2}+x^3(x+1)^{n-4}+ x^{4}(1+x)^{n-4}.$$
 This topology is obtained as a disjoint union of
 $\mathcal{P}(X_{n-4})$ and the topology having the diagram below.

 \unitlength=1truemm
\begin{picture}(40,40)(60,-10)
\put(119.25,-25){$\bullet$}
 \put(120,-24){\line(-1,1){10}}
\put(109.25,-15){$\bullet$}
 \put(110,-14){\line(1,1){10}}
 \put(119.25,-5){$\bullet$}
 \put(120,-4){\line(1,-1){10}}
 \put(129.25,-15){$\bullet$}
\put(120,-24){\line(1,1){10}}
\put(120,-4){\line(0,1){10}}%
\put(120,6){\line(0,1){10}}%
 \put(122,14){$X_4$}
\put(119.25,5){$\bullet$}
 \put(119,-29){$\emptyset$}
 \put(119.25,15){$\bullet$}
 \put(115,-35){Figure 2}
\end{picture}
\par\vspace{2.5cm}
\end{proof}
\begin{theorem}
 If $\tau$ is a topology obtained from  $(n-2)$ singletons as minimal open sets, with  the adjoining of $\{a,\;x_1\},\;  \{b,\;x_1\}$
  Then $|\tau|=10\cdot2^{n-4}$, and  the open polynomial of this topology is given by
$$ P_{\tau}(x)  =(x+1)^{n-2}+2x^2(x+1)^{n-3}+ x^3(1+x)^{n-3}=(1+x+2x^2+x^3)(1+x)^{n-3}.$$
   \end{theorem}
\begin{proof}
If we adjoin the two open  sets $\{a,\;x_1\},\;  \{b,\;x_1\}$ to the
discrete topology $\mathcal{P}(X_{n-2})$,  we obtain a topology of
$10.2^{n-4}$ open sets, its open polynomial is
$$ P_{\tau}(x)  =(x+1)^{n-2}+2x^2(x+1)^{n-3}+ x^3(1+x)^{n-3}.$$
Here too,  this topology is the disjoint union of
$\mathcal{P}(X_{n-3})$ and the topology $\tau'= \{\phi,\; a,\; b,\;
ab,\; X_3 \}$ on the set $X_3$ depicted below.
\begin{center}
$$ \unitlength=1truemm
\begin{picture}(40,40)(25,-13)
 \put(39.25,0){$\bullet$}
 \put(40,1){\line(-1,1){10}}
\put(29.25,10){$\bullet$} 
 \put(30,11){\line(1,1){10}}
 \put(39.25,20){$\bullet$}
 \put(39,23){$X_3$}
 \put(40,21){\line(1,-1){10}}
 \put(49.25,10){$\bullet$}
\put(40,1){\line(1,1){10}} \put(40,1){\line(0,-1){10}} 
 \put(39,-14){$\emptyset$}
\put(39.25,-10){$\bullet$}  \put(30,-20) {Figure 3}
\end{picture}
$$
\end{center}

\end{proof}
In the following result, we give the open set polynomials, when
adjoining relatively large open sets.
\begin{theorem}
 If $\tau$ is a topology obtained from  $(n-2)$ singletons as minimal open sets, by
 adjoining of $\{a,\;x\},\;{\rm and} \;  \{b,\;x,\;x_1, x_2, \ldots, x_j\}, \;1 \leq j \leq n-3 $
  Then $|\tau|=6\cdot2^{n-4}+2^{n-2-j}, \; 1 \leq j \leq n-3$, and the open polynomial of this topology is given by
$$ P_j(x)  =(x+1)^{n-2}+x^2(x+1)^{n-3}+ x^{j+2}(1+x)^{n-3-j}+x^{j+3}(1+x)^{n-3-j}.$$
   \end{theorem}
\begin{proof}
The  adjoining of the open  sets $\{a,\;x\}$ contributes
$\binom{n-3}{j}, 0 \leq j \leq n-3$  open sets of cardinality $j$.
By the same, the  contribution of  the open set $ \{b,\;x,\;x_1,
x_2, \ldots, x_j\}, \;1 \leq j \leq n-3 $, is $\binom{n-j-1}{j+i}, 0
\leq i \leq n-3-j$  open sets of cardinality $ j+i$. So the open
polynomial is, as stated $\displaystyle P_j(x)
=(x+1)^{n-2}+x^2(x+1)^{n-3}+
x^{j+2}(1+x)^{n-3-j}+x^{j+3}(1+x)^{n-3-j}.$
\end{proof}

\begin{theorem}
 If $\tau$ is a topology obtained from  $(n-2)$ singletons as minimal open sets, by
 adjoining  $\{a,\;x_1\},\;{\rm and} \;  \{b,\;x_2\} $,
  then $|\tau|=9\cdot2^{n-4}$.  The open polynomial of this topology is given by
$$ P(x) =(x+1)^{n-2}+2x^2(x+1)^{n-3}+ x^{4}(1+x)^{n-4}=(1+2x+3x^2+2x^3+x^4)(1+x)^{n-4}.$$
   \end{theorem}
\begin{proof}
The  adjoining of the open  sets $\{a,\;x_1\}$ contributes
$\binom{n-3}{j}, 0 \leq j \leq n-3$  open sets of cardinality $j$.
This is the same for the open set $ \{b,\; x_2\} $. Another
contribution comes from their union  $\{a,\;b, \; x_1, \; x_2\}$, it
is $\binom{n-4}{j+4}, 0 \leq j \leq n-4$. So, the wanted polynomial
is
 $ \displaystyle P_j(x)  =(x+1)^{n-2}+2x^2(x+1)^{n-3}+
x^{4}(1+x)^{n-4}.$
\end{proof}
 \begin{theorem}
 If $\tau$ is a topology obtained from  $(n-2)$ singletons as minimal open sets, by
 adjoining of $\{a,\;x \},\;{\rm and} \;  \{b, \;x_1,\; x_2, \; \ldots, x_j\}, \;2 \leq j \leq n-3
 $, then $|\tau|=6\cdot2^{n-4}+3\cdot2^{n-2-j}, \; 1 \leq j \leq n-3$ and the open set polynomials of these topologies are
$$ P_j(x)  =(x+1)^{n-2}+x^2(x+1)^{n-3}+ x^{j+1}(1+x)^{n-2-j}+x^{j+3}(1+x)^{n-j-3}.$$
   \end{theorem}
\begin{proof}
The  adjoining of the open  sets $\{a,\;x\}$ contributes by
$\binom{n-3}{j}, 0 \leq j \leq n-3$  open sets of cardinality $j$.
The  contribution of  the open set $ \{b,\; x_1,\; x_2,\; \ldots,
x_j\}, \;2 \leq j \leq n-3 $, is $\binom{n-j-1}{j+i}, 0 \leq i \leq
n-3-j$  open set of cardinality $ j+i$ .
\end{proof}
In the following result, each adjoined set contributes the same
numbers of open sets
\begin{theorem}
 If $\tau$ is a topology obtained from  $(n-2)$ singletons as minimal open sets, by
 adjoining of $\{a,\;x_1, \; x_2\}\;{\rm and} \;  \{b,\;x_1, x_2 \} $
  Then $|\tau|=7\cdot2^{n-4}$, and its open polynomial is
$$ P(x)  =(x+1)^{n-2}+2x^3(x+1)^{n-4}+ x^{4}(1+x)^{n-4}.$$
   \end{theorem}
\begin{proof}
The  adjoining of any of  the open  sets $\{a,\;x_1, x_2\}$,
$\{b,\;x_1, x_2\}$, or $\{a,\;b, \;, x_1, x_2, \}$ contribute
$2^{n-4}$ open sets to the topology. So, the open polynomial.
\end{proof}

\begin{theorem}
 If $\tau$ is a topology obtained from  $(n-2)$ singletons as minimal open sets, by
 adjoining of $\{a,\;x_1, \; x_2\}\;{\rm and} \;  \{b,\;x_1, x_3 \} $
  Then $|\tau|=13\cdot2^{n-5}$, and its open polynomial is
$$ P(x)  =(x+1)^{n-2}+2x^3(x+1)^{n-4}+ x^{5}(1+x)^{n-5}.$$
   \end{theorem}
\begin{proof}
The  adjoining of the open  sets $\{a,\;x_1, x_2\}$ and $\{b,\;x_1,
x_3\}$, give raise to  $2^{n-4}$ open sets. Their union $\{a,\;b, \;
x_1,\; x_2, \; x_3, \}$ gives  $2^{n-5}$ other open sets. So, the
open polynomial is as stated above.
\end{proof}

\begin{theorem}
 If $\tau$ is a topology obtained from  $(n-2)$ singletons as minimal open sets, by
the  adjoining of $\{a,\;x_1, \; x_2\}\;{ and} \; \{b,\;x_1,x_2,\;
x_3 \} .$
  Then $|\tau|=6\cdot2^{n-4}$, and its open polynomial is
$$ P(x)  =(x+1)^{n-2}+x^3(x+1)^{n-4}+ x^{4}(1+x)^{n-5}+x^{5}(1+x)^{n-5} .$$
   \end{theorem}
\begin{proof}
The  adjoining of the open  sets $\{a,\;x_1,\; x_2\}$ contributes
$2^{n-4}$ open sets, and $\{b,\;x_1, \;x_2,\; x_3\}$, contributes by
$2^{n-5}$  to the topology. Their union $\{a,\;b, \; x_1,\; x_2, \;
x_3 \}$ gives raise to $2^{n-5}$. So, the open polynomial is as
stated above.
\end{proof}
\begin{theorem}
 If $\tau$ is a topology obtained from  $(n-2)$ singletons as minimal open sets, by
 adjoining of $\{a,\;x_1, \; x_2\}\;{\rm and} \;  \{b,\;x_3,\; x_4 \} $
  Then $|\tau|=25\cdot2^{n-6}$, and its open polynomial is
$$ P(x)  =(x+1)^{n-2}+2x^3(x+1)^{n-4}+ x^{6}(1+x)^{n-6}.$$
   \end{theorem}
\begin{proof}
The  adjoining of the open  sets $\{a,\;x_1,\; x_2\}$ and $\{
b,\;x_3,\; x_4\}$, contributes  $2^{n-4}$ open sets each one to the
topology. Their union $\{a,\;b, \; x_1,\; x_2, \; x_3,\;x_4 \}$
gives raise to $2^{n-6}$. Furthermore, the  polynomial is
 $\displaystyle P(x)=(x+1)^{n-2}+2x^3(x+1)^{n-4}+ x^{6}(1+x)^{n-6}.$
\end{proof}

The open polynomials of the topologies with $(n-3)$ and $(n-4)$ and
in general $n-i$ minimal open sets, and cardinality $ \geq
6\cdot2^{n-4}$, are investigated in the following theorems. The case
$(n-3)$ has more than 10 cases. We give the open set polynomial
according to the cardinality of the topology $\tau$.
\begin{theorem}
 If $\tau$ is a topology having  $(n-3)$ singletons as minimal open sets and
 $|\tau|=6\cdot2^{n-4}$,
 then $\tau$  is one of the following : \\
$1)$ Adjunction of  the sets $\{a,\;x_1\}$ and $\{a, \; b,\; x_1\}$
and $\{a, \; c, \;x_1\}$ to $\mathcal{P}(X_{n-3})$. Its open
polynomial is then
$$ P(x)=(x+1)^{n-3}+x^2(x+1)^{n-4}+2 x^{3}(1+x)^{n-4}+x^4(x+1)^{n-4}.$$
 $ 2)$ Adjunction
of the sets $\{a,\;x_1\}$, $\{b,\;x_2\}$ and $\{a, \; c, \;
 x_1\}$ to $\mathcal{P}(X_{n-3})$. Its polynomial is
 $$ P(x)  =(x+1)^{n-3}+2x^2(x+1)^{n-4}+ x^{3}(1+x)^{n-4}+x^4(x+1)^{n-5}+ x^5(x+1)^{n-5}.$$
$ 3) $Adjunction of the sets $\{a,\;x_1\}$, $\{b,\;x_1\}$ and $\{a,
\;b, \;  c, \; x_1\}$, we obtain the open polynomial
$$ P(x)  =(x+1)^{n-3}+2x^2(x+1)^{n-4}+ x^{3}(1+x)^{n-4}+x^4(x+1)^{n-4}.$$
$ 4)$  Adjunction of the 3 sets $\{a,\;x_1\}$, $\{b,\;x_1\}$ and
$\{a,\;  c, \; x_1, \; x_2 \}$, yields the polynomial
$$ P(x)  =(x+1)^{n-3}+2x^2(x+1)^{n-4}+ x^{3}(1+x)^{n-4}+x^4(x+1)^{n-5}+ x^5(x+1)^{n-5}.$$
$5)$ The adjunction of  $\{a,\;x_1\}$ and $\{b,\;x_2\}$ and  $\{c,
\; x_1, \; x_3\}$, gives the  polynomial $$ P(x)
=(x+1)^{n-3}+2x^2(x+1)^{n-4}+ x^{3}(1+x)^{n-5}+2x^4(x+1)^{n-5}+
x^5(x+1)^{n-6}+x^6(x+1)^{n-6}.$$
 $6)$ The sets $\{a,\;x_1\}$, $\{b,\;x_1\}$ and  $\{c, \;
x_1, \; x_2, \; x_3\}$ adjoined to $\mathcal{P}(X_{n-3})$ give a
topology with open  polynomial
$$ P(x)  =(x+1)^{n-3}+2x^2(x+1)^{n-4}+
x^{3}(1+x)^{n-5}+x^4(x+1)^{n-5}+ 2x^5(x+1)^{n-6}+x^6(x+1)^{n-6}.$$
   \end{theorem}
\begin{proof}
We prove just the first case, the others are similar. The  adjoining
of  $\{a,\;x_1\}$ gives $2^{n-4}$ open sets. The sets   $\{a, \;
b,\; x_1\}$ and $\{a, \; c, \;x_1\}$  each one, also   gives
$2^{n-4}$ open sets. We obtain the open polynomial of the topology
 $$ P(x)=(x+1)^{n-3}+2x^2(x+1)^{n-4}+ x^{3}(1+x)^{n-4}+x^4(x+1)^{n-4}.$$
\end{proof}
The case where the topology has $|\tau|=7\cdot2^{n-4}$ open sets is
given in this result
\begin{theorem}
 If $\tau$ is a topology having  $(n-3)$ singletons as minimal open sets, and  $|\tau|=7\cdot2^{n-4}$,
 then  $\tau$ comes either by \\
 $1)$ Adjoining  the sets $\{a,\;x_1\}$, $\{b,\;x_1\}$ and $\{a, \; c, \;
 x_1\}$ to $\mathcal{P}(X_{n-3})$  in this case its open set polynomial is
$$ P(x)=(x+1)^{n-3}+2x^2(x+1)^{n-4}+ 2x^{3}(1+x)^{n-4}+x^4(x+1)^{n-4}.$$
 $2)$ Adjoining the open sets $\{a,\;x_1\}$, $\{b,\;x_1\}$ and $\{c, \; x_1, \;
 x_2\}$, and then
$$ P(x)=(x+1)^{n-3}+2x^2(x+1)^{n-4}+ x^{3}(1+x)^{n-4}+x^{3}(1+x)^{n-5}+2x^4(x+1)^{n-5}+x^{5}(1+x)^{n-5}.$$

   \end{theorem}
\begin{proof}
The  adjoining of the open  sets $\{a,\;x_1 \}$ and $\{b,\;x_1\}$
contribute by $2^{n-4}$ open sets each one to the topology. Their
union $\{a,\;b, \; x_1 \}$, as well as  $\{a, \; c, \;
 x_1\}$ give raise to $2^{n-4}$ each one.
The open set, $\{a,b,c,x_1\}$  also gives  $2^{n-4}$ open sets. So,
the open polynomial is
 $$ P(x)=(x+1)^{n-3}+2x^2(x+1)^{n-4}+ 2x^{3}(1+x)^{n-4}+x^4(x+1)^{n-4}.$$
\end{proof}
The remaining cases of the topologies with $(n-3)$ open sets as
singletons, are gathered in this theorem.
\begin{theorem}
 If $\tau$ is a topology having  $(n-3)$ singletons as minimal open sets,
 then  $\tau$ is obtained \\
$1)$ By adjoining  the sets $\{a,\;x_1\}$,  $\{ b, \; x_1\}$ and $\{
c, \;x_1\}$ to  $\mathcal{P}(X_{n-3})$, $|\tau|=9\cdot2^{n-4}$ with
open polynomial
$$ P(x)=(x+1)^{n-3}+3x^2(x+1)^{n-4}+3 x^{3}(1+x)^{n-4}+x^4(x+1)^{n-4}.$$
 $ 2)$ By adjoining the open sets
of the sets $\{a,\;x_1\}$,  $\{b,\;x_1\}$ and $\{ c, \;
 x_2\}$, then  $|\tau|=15\cdot2^{n-5}$ and
 $$ P(x)  =(x+1)^{n-3}+3x^2(x+1)^{n-4}+ x^{3}(1+x)^{n-4}+2x^4(x+1)^{n-5}+ x^5(x+1)^{n-5}.$$
$ 3)$ By the adjunction of the sets $\{a,\;x_1\}$,  $\{b,\;x_2\}$
and $\{ c, \; x_3\}$, $|\tau|=27\cdot2^{n-6}$ with open polynomial
$$ P(x)  =(x+1)^{n-3}+3x^2(x+1)^{n-4}+ 3x^{4}(1+x)^{n-5}+x^6(x+1)^{n-6}.$$
$ 4)$  By the adjunction of the sets $\{a,\;x_1\}$, $\{b,\;x_2\}$
and $\{ c, \; x_1, \; x_2 \}$, we obtain $|\tau|=13\cdot2^{n-5}$ and
$$ P(x)  =(x+1)^{n-3}+2x^2(x+1)^{n-4}+ x^{3}(1+x)^{n-5}+ 3x^{4}(1+x)^{n-5}+ x^5(x+1)^{n-5}.$$
$5)$ By the adjoining of  $\{a,\;x_1\}$, $\{b,\;x_1\}$ and  $\{c, \;
x_2, \; x_3\}$, we obtain a topology with $|\tau|=25\cdot2^{n-6}$
with open polynomial
$$ P(x) =(x+1)^{n-3}+2x^2(x+1)^{n-4}+x^{3}(1+x)^{n-4}+
x^{3}(1+x)^{n-5}+ 2x^5(x+1)^{n-6}+x^6(x+1)^{n-6}.$$

   \end{theorem}

\begin{proof}
Here too, just the first case suffices, the other proofs are
similar. The adjoining of each of the three open sets $\{y,\;x_1\},
y=a,b,c$ contributes  $2^{n-4}$ open sets to the topology. Their
unions  three times $2^{n-4}$. So, the cardinality of the topology
is $9\cdot2^{n-4}$, and its polynomial is
$$ P(x)=(x+1)^{n-3}+3x^2(x+1)^{n-4}+3 x^{3}(1+x)^{n-4}+x^4(x+1)^{n-4}.$$
\end{proof}

In what follows, we give the open polynomials of the topologies
having $(n-4)$ singletons as minimal open sets. We have in all five
cases, i.e; these topologies are obtained by adjoining some open
sets to the discrete topology on $(n-4)$ elements.

\begin{theorem}
 If $\tau$ is a topology, $|\tau| \geq 6\cdot2^{n-4}$  and having  $(n-4)$ singletons as minimal open sets,
 then it obtained by adjoining to $\mathcal{P}(X_{n-4})$ one kind of the open sets below:  \\
 $1)$ Either   $ \displaystyle \{a, \;x_1\}, \;  \{b, \;x_1\},\;\{c, \;x_1 \},\;
\{ d, \; x_1 \}$, and in this case, its open polynomial is  $$
P_1(x) =x(x+1)^{n-1}+(x+1)^{n-5},\; {\rm
and}\;|\tau|=17\cdot2^{n-5}.
$$
 $2)$ Or   $ \displaystyle \{a, \; x_1\}, \;
\{b, \; x_1\},\;\{c, \; x_1 \},\;\{ d, \; x_2 \}$, in  this case
$$ P_2(x) =x^3(x+1)^{n-3}+x(x+1)^{n-2}+(1+x+x^2)(x+1)^{n-6},\; {\rm
and}\;|\tau|=27\cdot2^{n-6}.$$
 $3)$ Or   $ \displaystyle \{a, \;x_1\},
\;  \{b, \;x_1\},\;\{c,\;x_1\},\; \{d, \; x_1,\;x_2 \}$, in this
case
$$ P_3(x) =x^3(x+1)^{n-3}+x(x+1)^{n-2}+(x+1)^{n-5},\; {\rm and}\;
|\tau|=13\cdot2^{n-5}.$$ $ 4)$ Or   $\{a,\; x_1\}, \; \{b,
\;x_1\},\;\{c, \;x_2 \},\; \{d, \; x_2 \}$, in  this case $$ P_4(x)
=x^2(x+1)^{n-2}+x^2(x+1)^{n-5}+(1+2x+3x^2+x^3)(x+1)^{n-6},\; {\rm
and}\;|\tau|=25\cdot2^{n-6}. $$
 $5)$  Or $\{a, \;x_1 \}, \;
\{b, \;x_1\},\;\{c, \;x_1 \},\; \{d, \; x_1,\;x_2 \}$,  in this case
$$
P_5(x) =x^2(x+1)^{n-2}+x^2(x+1)^{n-4}+x^2(x+1)^{n-5}+(x+1)^{n-4},\;
{\rm and}\;|\tau|=13\cdot2^{n-5}. $$
\end{theorem}

\begin{proof}
The  adjoining of each of the four  open  sets $\{y,\;x_1\},
y=a,\;b,\;c,\; d$ contributes by $2^{n-5}$ open sets,  their unions
 $16$ times $2^{n-5}$. So, the cardinality of the
topology is $17.2^{n-4}$, and its polynomial is as stated. The other
cases are treated similarly.
\end{proof}

\begin{theorem}
 If $\tau$ is a topology having  $(n-i)$ singletons as minimal open sets, $5\leq i \leq n-2$,
 by adjoining the open sets $\{x_1,\;y_1\}, \;  \{x_1,\;y_2\},\;...,\;
\{ x_1,\;y_i \}, \;$, to the discrete topology
$\mathcal{P}(X_{n-i})$  then $|\tau|=2^{n-1}+2^{n-i-1}\texttt{}$,
and its open polynomial is
$$ P(x)  =x(x+1)^{n-1}+(x+1)^{n-i-1}.$$
By adjoining the open sets $\{x_1,\;y_1\}, \; \{x_1,\;y_2\},\;...,\;
\{ x_1,\;y_{i-1} \}, \{ x_2,\;y_i \} \;$,
  $\tau$ is a topology with $|\tau|=6.2^{n-4}+3.2^{n-i-2}$
  and its open polynomial is
$$ P(x)  =x^3(x+1)^{n-3}+x(x+1)^{n-2}+(1+x+x^2)(1+x)^{n-i-2}.$$
Finally, if $\tau$ is obtained by adjoining $\{x_1,\;y_1\}, \;
\{x_1,\;y_2\},\;...,\; \{ x_1,\;y_{i-1} \}, \{x_1,\; x_2,\;y_i \}
\;$, or $\{x_1,\;y_1\}, \; \{x_1,\;y_2\},\;...,\; \{ x_1,\;y_{i-1}
\}, \{x_1,\; y_1,\;y_i \} \;$,
  then $|\tau|=6.2^{n-4}+2^{n-i-1}$, and its open polynomial is

$$ P(x)  =x^3(x+1)^{n-3}+x(x+1)^{n-2}+(1+x)^{n-i-1}.$$
   \end{theorem}

\begin{proof}
The  adjoining of each of the three open  sets $\{x_1, \;y_j\},
j=1,\;...\; i$ contributes  $2^{n-i-1}$, their different unions give
$\binom {i}{j} 2^{n-i-1}$. So, the  open set polynomial is
\begin{eqnarray*}
       P(x) & =& (x+1)^{n-i}+ x(x+1)^{n-i-1} \left ((1+x)^{i}-1\right )  \\
       P(x) &= & x(x+1)^{n-1}+(x+1)^{n-i-1}.
     \end{eqnarray*}
The other cases are treated similarly.
\end{proof}
As a final result, we state
\begin{theorem}
 If $\tau$ is a topology on the set $X_n$, such that $|\tau| \geq 6 \cdot
 2^{n-4}$, then its
  open set  polynomial is unimodal.
   \end{theorem}

\bigskip
\hrule
\bigskip

\noindent 2000 {\it Mathematics Subject Classification}:
Primary 11B39; Secondary 11B75.\ \

\noindent \emph{Keywords: Open set polynomial, log--concave
sequence, polynomial with real zeros,topology,  unimodal sequence.}

\end{document}